\documentclass{amsproc}
\usepackage{amscd}

\newtheorem{theorem}{Theorem}[section]

\newtheorem{lemma}[theorem]{Lemma}
\newtheorem{proposition}[theorem]{Proposition}

\theoremstyle{definition}
\newtheorem{definition}[theorem]{Definition}
\newtheorem{notation}[theorem]{Notation}
\newtheorem{example}[theorem]{Example}

\theoremstyle{remark}
\newtheorem{remark}[theorem]{Remark}

\numberwithin{equation}{section}

 \DeclareMathOperator{\Hom}{Hom}
 \DeclareMathOperator{\id}{id}
 \DeclareMathOperator{\im}{im}
 \DeclareMathOperator{\rank}{rank}
 \DeclareMathOperator{\corank}{corank}
 \DeclareMathOperator{\terminus}{terminus}

\renewcommand{\d}{\partial}
\newcommand{\e}{\epsilon}

\renewcommand{\L}{\Lambda}
\newcommand{\Or}{\mathcal{O}}

\newcommand{\w}{\wedge}
\newcommand{\Z}{\mathbf{Z}}

\begin{document}

\title{The algebraic structure of the universal complicial sets}

\author{Richard Steiner}

\address{School of Mathematics and Statistics, University of
Glasgow, University Gardens, Glasgow, Great Britain G12 8QW}

\email{Richard.Steiner@glasgow.ac.uk}



\begin{abstract}
The nerve of a strict omega-category is a simplicial set with
additional structure, making it into a so-called complicial
set, and strict omega-categories are in fact equivalent to
complicial sets. The nerve functor is represented by a sequence
of strict omega-categories, called orientals, which are
associated to simplexes. In this paper we give a detailed
algebraic description of the morphisms between orientals. The
aim is to describe complicial sets algebraically, by operators
and equational axioms.
\end{abstract}

\maketitle

\section{Introduction} \label{S1}

The orientals or oriented simplexes are a sequence of strict
$\omega$-categories
$$\Or_0,\Or_1,\ldots$$
associated to simplexes. They were discovered by Street, who
described them as fundamental objects in nature~\cite{Str}. A
strict $\omega$-category~$X$ has a \emph{nerve}, consisting of
the sequence of morphism sets
$$\Hom(\Or_0,X),\Hom(\Or_1,X),\ldots;$$
Verity~\cite{V} has shown that the nerve functor makes the
category of strict $\omega$-categories equivalent to a category
of simplicial sets with additional structure, called
\emph{complicial sets}.

By definition, a complicial set is a simplicial set with a
distinguished class of elements, called \emph{thin elements},
subject to certain axioms; Verity's theorem therefore amounts
to a description of strict $\omega$-categories in combinatorial
terms. There is an analogous cubical theory~\cite{ABS} which
gives a more algebraic description, in terms of cubical sets
with additional operations and equational axioms. This paper is
part of a programme aimed at producing a similar algebraic
description for complicial sets, using operations and
equational axioms rather than distinguished subsets.

In this paper we consider the universal examples; in other
words, we consider the nerves of the orientals themselves. This
is in fact a purely algebraic problem. The category of
orientals can be embedded in the category of chain complexes
and chain maps~\cite{Ste}: the objects are the chain complexes
of the standard simplexes; the morphisms are the
augmentation-preserving chain maps taking standard basis
elements to sums of standard basis elements. This gives a
simple algebraic description of the nerves of orientals, as
subsets of graded abelian groups, but we really need an
internal description independent of any supersets. We will
therefore solve the following problem: find an algebraic
structure on graded sets, consisting of internal operations and
equational axioms, such that the nerve of~$\Or_n$ is freely
generated by its identity endomorphism~$\iota_n$. The
equational axioms which solve this problem will be called
\emph{complicial identities;} they are listed in
Definition~\ref{D6.1}.

The structure of the paper is as follows. In Section~\ref{S2}
we recall the description of orientals in terms of chain maps.
In Section~\ref{S3} we describe the additional operations (they
were introduced with different notation and terminology
in~\cite{Ste}). In Section~\ref{S4} we show that the nerve
of~$\Or_n$ is generated by~$\iota_n$ (see Theorem~\ref{T4.12}); this was also done
in~\cite{Ste}, but here we give more precise details and in
effect obtain canonical forms for the elements of the nerve (see Theorem~\ref{T4.11}). In
Section~\ref{S5} we give some additional properties of the
nerves, for later use. In Section~\ref{S6} we describe the
complicial identities, and we show that they are satisfied in
the nerves of orientals. In Section~\ref{S7} we describe
certain consequences of the complicial identities, and in
Section~\ref{S8} we prove the main theorem
(Theorem~\ref{T8.7}), showing that the nerve of~$\Or_n$ is the
set with complicial identities freely generated by~$\iota_n$.

\section{Orientals and chain complexes} \label{S2}

From~\cite{Ste} we recall the description of the category of
orientals in terms of chain maps. For $n=0,1,\ldots\,$ let
$\Z\Delta(n)$ be the cellular chain complex of the standard
$n$-simplex. We regard $\Z\Delta(n)$ as a free graded abelian
group with a prescribed basis. The basis elements, written in
the form $[a_0,\ldots,a_q]$, correspond to the $(q+1)$-tuples
of integers $a_0,\ldots,a_q$ such that
$$0\leq a_0<a_1<\ldots<a_q\leq n;$$
a basis element $[a_0,\ldots,a_q]$ is homogeneous of
degree~$q$. The boundary homomorphism
$\d\colon\Z\Delta(n)\to\Z\Delta(n)$, which lowers degrees
by~$1$, is given on basis elements of positive degree by
$$\d[a_0,\ldots,a_q]
 =\sum_{i=0}^q(-1)^i [a_0,\ldots,a_{i-1},a_{i+1},\ldots,a_q].$$
There is also an augmentation homomorphism
$\e\colon\Z\Delta(n)\to\Z$, which is given on basis elements of
degree~$0$ by
$$\e[a_0]=1,$$
and which vanishes on basis elements of positive degree.

We will write $\Z\Delta(m,n)$ for the abelian group consisting
of the chain maps from $\Z\Delta(m)$ to $\Z\Delta(n)$; thus a
member of $\Z\Delta(m,n)$ is a degree-preserving abelian group
homomorphism $f\colon\Z\Delta(m)\to\Z\Delta(n)$ such that $\d
f=f\d$. Ordinary function composition makes the groups $\Z\Delta(m,n)$ into the morphism sets of a category $\Z\Delta$ with objects $0,1,2,\ldots\,$. For each fixed target~$n$, it is convenient to regard the sequence
$$\Z\Delta(0,n),\ \Z\Delta(1,n),\ \ldots$$
as a graded abelian group $\Z\Delta(-,n)$. 

We will also write $\Or(m,n)$ for the subset of
$\Z\Delta(m,n)$ consisting of the chain maps~$f$ which are
augmentation-preserving ($\e f=\e$) and which take basis
elements to sums of basis elements (if $a$~is a basis element
then $f(a)=b_1+\ldots+b_k$ for some $k\geq 0$ and for some
basis elements $b_1,\ldots,b_k$). As before, we get a category~$\Or$ with morphism sets $\Or(m,n)$, and we get graded sets $\Or(-,n)$. It follows from~\cite{Ste} that there are natural bijections
$$\Or(m,n)\cong\Hom(\Or_m,\Or_n),$$
where $\Hom(\Or_m,\Or_n)$ is the set of strict $\omega$-category morphisms from the oriental~$\Or_m$ to the oriental~$\Or_n$; the category~$\Or$ will therefore be called the \emph{category of orientals}. Because of the bijections, the nerves of the orientals~$\Or_n$ can be identified with the graded sets $\Or(-,n)$. We will describe the nerves by using these identifications.

\section{Operations in the nerves of orientals} \label{S3}

In this section we construct three families of operations in
$\Or(-,n)$: \emph{face operations}, \emph{degeneracy
operations} and \emph{wedge operations}. They will be
restrictions of operations in $\Z\Delta(-,n)$. The face and
degeneracy operations come from the standard simplicial set
structure on $\Z\Delta(-,n)$, and we begin by recalling the
definition of a simplicial set.

\begin{definition} \label{D3.1}
A \emph{simplicial set}~$X$ is a sequence of sets
$X_0,X_1,\ldots\,$ together with face operations
$$\d_i\colon X_m\to X_{m-1}\quad (m>0,\ 0\leq i\leq m)$$
and degeneracy operations
$$\e_i\colon X_m\to X_{m+1}\quad (0\leq i\leq m)$$
such that
\begin{align*}
 &\d_i\d_j=\d_{j-1}\d_i\quad (i<j),\\
 &\d_i\e_j=\e_{j-1}\d_i\quad (i<j),\\
 &\d_i\e_i=\d_{i+1}\e_i=\id,\\
 &\d_i\e_j=\e_j\d_{i-1}\quad (i>j+1),\\
 &\e_i\e_j=\e_{j+1}\e_i\quad (i\leq j).
\end{align*}
\end{definition}

It is well known that simplicial set operations $X_m\to X_p$ correspond to non-decreasing functions
$$\{0,1,\ldots,p\}\to\{0,1,\ldots,m\}$$
between sets of integers; a face operation $\d_i\colon X_m\to X_{m-1}$ corresponds to the function~$\d_i^\vee$ with 
$$\d_i^\vee(0)=0,\ \ldots,\ \d_i^\vee(i-1)=i-1,\ 
 \d_i^\vee(i)=i+1,\ \ldots,\ \d_i^\vee(m-1)=m,$$
and a degeneracy $\e_i\colon X_m\to X_{m+1}$ corresponds to the function~$\e_i^\vee$ with
$$\e_i^\vee(0)=0,\ \ldots,\ \e_i^\vee(i)=\e_i^\vee(i+1)=i,\ \ldots,\ \e_i^\vee(m+1)=m.$$
We will now describe the corresponding operations in $\Z\Delta(-,n)$ in terms of the basis elements $[a_0,\ldots,a_q]$ for the chain groups $\Z\Delta(p)$. For these basis elements we will use notations such as $[\mathbf{b},\mathbf{c}]$ or
$[\mathbf{b},i,\mathbf{c}]$, where
$\mathbf{b}$~and~$\mathbf{c}$ are suitable sequences of
integers. The outcome is as follows.

\begin{definition} \label{D3.2}
Let $x$ be a chain map in $\Z\Delta(m,n)$.

For $m>0$ and $0\leq i\leq m$, the \emph{face} $\d_i x$ is the
chain map in $\Z\Delta(m-1,n)$ given on basis elements by
$$(\d_i x)[\mathbf{b},\mathbf{c}]=x[\mathbf{b},\mathbf{c}'],$$
where the terms of~$\mathbf{b}$ are less than~$i$, the terms
of~$\mathbf{c}$ are greater than or equal to~$i$, and the terms
of~$\mathbf{c}'$ are obtained from those of~$\mathbf{c}$ by
adding~$1$.

For $0\leq i\leq m$ the \emph{degeneracy} $\e_i x$ is the chain
map in $\Z\Delta(m+1,n)$ given on basis elements by
\begin{align*}
 &(\e_i x)[\mathbf{b},\mathbf{c}]=x[\mathbf{b},\mathbf{c}''],\\
 &(\e_i x)[\mathbf{b},i,\mathbf{c}]
 =(\e_i x)[\mathbf{b},i+1,\mathbf{c}]
 =x[\mathbf{b},i,\mathbf{c}''],\\
 &(\e_i x)[\mathbf{b},i,i+1,\mathbf{c}]=0,
\end{align*}
where the terms of~$\mathbf{b}$ are less than~$i$, the terms
of~$\mathbf{c}$ are greater than $i+1$, and the terms
of~$\mathbf{c}''$ are obtained from those of~$\mathbf{c}$ by
subtracting~$1$.
\end{definition}

We get the following result.

\begin{proposition} \label{P3.3}
The face and degeneracy operations
$$\d_i\colon\Z\Delta(m,n)\to\Z\Delta(m-1,n),\quad
 \e_i\colon\Z\Delta(m,n)\to\Z\Delta(m+1,n)$$
are group homomorphisms making $\Z\Delta(-,n)$ into a
simplicial set. They restrict to operations
$$\d_i\colon\Or(m,n)\to\Or(m-1,n),\quad
 \e_i\colon\Or(m,n)\to\Or(m+1,n)$$
making $\Or(-,n)$ into a simplicial set.
\end{proposition}

\begin{proof}
It is clear that the operations in $\Z\Delta(-,n)$ are
homomorphisms satisfying the simplicial identities. If
$x\in\Or(-,n)$, so that $x$~is augmentation-preserving and
takes basis elements to sums of basis elements, then $\d_i x$
and $\e_i x$ clearly belong to $\Or(-,n)$ as well.
\end{proof}

Degeneracies can be characterised as follows.

\begin{proposition} \label{P3.4}
Let $x$ be a morphism in $\Or(-,n)$. Then $x$~is in the image
of~$\e_i$ if and only if $xa=0$ for every basis element~$a$
including $i$~and $i+1$.
\end{proposition}

\begin{proof}
By definition, if $x$~is in the image of~$\e_i$ then
$x$~vanishes on every basis element including $i$~and $i+1$.

Conversely, suppose that $x$~vanishes on every basis element
including $i$~and~$i+1$. Since $x$~is a chain map,
$$x[\mathbf{b},i,\mathbf{c}]=x[\mathbf{b},i+1,\mathbf{c}]$$
for all basis elements of the form
$[\mathbf{b},i,i+1,\mathbf{c}]$, and it follows that
$x=\e_i\d_i x$.
\end{proof}

Next we define the wedge operations.

\begin{definition} \label{D3.5}
Let $m$~and~$i$ be integers with $0\leq i\leq m-1$. If
$x$~and~$y$ are chain maps in $\Z\Delta(m,n)$ such that $\d_i
x=\d_{i+1}y$, then the \emph{wedge} $x\w_i y$ is the chain map
in $\Z\Delta(m+1,n)$ given by
$$x\w_i y=\e_{i+1}x-\e_i^2\d_{i+1}y+\e_i y.$$
\end{definition}

One can think of the partial binary operation
$$(x,y)\mapsto \d_{i+1}(x\w_i y)$$ 
as a pasting composition, gluing two $m$-simplexes along a common $(m-1)$-di\-men\-sion\-al face to form a larger $m$-simplex. The wedge $x\w_i y$ itself is a higher-dimensional structure containing this composite. It also contains the factors, because of the following obvious result.

\begin{proposition} \label{P3.6}
If $x\w_i y$ is defined in $\Z\Delta(-,n)$ then
$$\d_i(x\w_i y)=y,\quad \d_{i+2}(x\w_i y)=x.$$
\end{proposition}

We also get the following results.

\begin{proposition} \label{P3.7}
Let $x$~and~$y$ be morphisms in $\Or(m,n)$ such that $\d_i
x=\d_{i+1}y$. Then $x\w_i y$ is a morphism in $\Or(m+1,n)$.
\end{proposition}

\begin{proof}
Given that $x$~and~$y$ are augmentation-preserving and that
they take basis elements to sums of basis elements, we must
show that $x\w_i y$ has the same properties. This is
straightforward.
\end{proof}

\begin{proposition} \label{P3.8}
Let $z$ be a morphism in $\Or(m+1,n)$, and let $i$ be an
integer with $0\leq i\leq m-1$. Then the following are
equivalent.

\textup{(1)} There are morphisms $x,y$ in $\Or(m,n)$ with $\d_i
x=\d_{i+1}y$ such that $z=x\w_i y$.

\textup{(2)} There are chain maps $u,v$ in $\Z\Delta(m,n)$ such
that $z=\e_i u+\e_{i+1}v$.

\textup{(3)} One has $za=0$ for every basis element~$a$
including $i,i+1,i+2$.
\end{proposition}

\begin{proof}
We show that $(1)\Rightarrow (2)\Rightarrow (3)\Rightarrow
(1)$.

If $z=x\w_i y$ then $z$~has the form $\e_i u+\e_{i+1}v$ by
definition.

If $z=\e_i u+\e_{i+1}v$ then clearly $za=0$ for every basis
element~$a$ including $i,i+1,i+2$.

Suppose that $za=0$ for every basis element~$a$ including
$i,i+1,i+2$. Let
$$x=\d_{i+2}z,\quad y=\d_i z.$$
It follows from Proposition~\ref{P3.3} that $x$~and~$y$ are
morphisms in $\Or(m,n)$ and that $\d_i x=\d_{i+1}y$; the wedge
$x\w_i y$ therefore exists. The morphisms $z$~and $x\w_i y$
then agree on basis elements not including~$i$, on basis
elements not including $i+2$, and on basis elements including
$i,i+1,i+2$. Since $z$~and $x\w_i y$ are chain maps, they must
also agree on basis elements including $i$~and $i+2$ but not
$i+1$. Therefore $z=x\w_i y$.

This completes the proof.
\end{proof}

\section{Canonical forms} \label{S4}

Let $\iota_n$ be the identity morphism in $\Or(n,n)$. In this
section we show that the elements of $\Or(-,n)$ can be
expressed in terms of~$\iota_n$ by using the face, degeneracy
and wedge operations. In effect we find canonical forms for the
morphisms in $\Or(-,n)$ (see Theorem~\ref{T4.11}). The argument
is based on the following result.

\begin{theorem} \label{T4.1}
Let $x$ be a morphism in $\Or(m,n)$. If $[i]$~is a zero-dimensional basis element in $\Z\Delta(m)$, then $x[i]$ is a zero-dimensional basis element $[x(i)]$ in $\Z\Delta(n)$. The integers $x(0),\ldots,x(m)$ are such that
$$0\leq x(0)\leq x(1)\leq\ldots\leq x(m)\leq n.$$
If $a=[a_0,\ldots,a_q]$ is a basis element in $\Z\Delta(m)$, then $xa$ is a sum of basis elements
$[b_0,\ldots,b_q]$ with 
$$x(a_0)\leq b_0<b_1<\ldots<b_q\leq x(a_q).$$
\end{theorem}

\begin{proof}
Let $[i]$ be a zero-dimensional basis element in $\Z\Delta(m)$.
Then $x[i]$ is a sum of zero-dimensional basis elements in
$\Z\Delta(n)$, and this sum has exactly one term because $x$~is
augmentation-preserving. Therefore $x[i]=[x(i)]$ for some
integer $x(i)$ with $0\leq x(i)\leq n$.

For $0<i\leq m$ we have
$$\d x[i-1,i]=x\d[i-1,i]=x([i]-[i-1])=[x(i)]-[x(i-1)].$$
But $x[i-1,i]$ is a sum of basis elements $[j,k]$, so $\d
x[i-1,i]$ is a sum of expressions $[k]-[j]$ with $j<k$.
Therefore $x(i-1)\leq x(i)$.

Now let $a=[a_0,\ldots,a_q]$ be a basis element for $\Z\Delta(m)$. We
will show by induction on~$q$ that $xa$ is a sum of basis
elements $[b_0,\ldots,b_q]$ with $x(a_0)\leq b_0$. The result is
clear when $q=0$, and is trivial when $xa=0$. From now on,
suppose that $q>0$ and $xa\neq 0$, so that $xa$ is a non-empty
sum of basis elements of positive dimension. Let
$[k_0,\ldots,k_q]$ be a term in this sum such that $k_0$~is as
small as possible and, subject to this condition, such that
$k_1$~is as large as possible. Then the basis element
$[k_0,k_2,\ldots,k_q]$ has a negative coefficient in $\d xa$
because there is no possibility of cancellation. Since $\d
xa=x\d a$, it follows that $[k_0,k_2,\ldots,k_q]$ is a term in
$x[a_0,a_1,\ldots,a_{i-1},a_{i+1},\ldots,a_q]$ for some odd value
of~$i$, and it then follows from the inductive hypothesis that
$x(a_0)\leq k_0$. Since $k_0$~is minimal, $xa$ is a sum of basis
elements $[b_0,\ldots,b_q]$ with $x(a_0)\leq b_0$.

A similar argument shows that $xa$ is a sum of basis elements
$[b_0,\ldots,b_q]$ with $b_q\leq x(a_q)$.

This completes the proof.
\end{proof}

Let $x$ be a morphism in $\Or(m,n)$, let $x[m]=[t]$, and let $r$ be minimal such that $x[r]=[t]$. The images of the zero-dimensional basis elements $[0],\ldots,[m]$ in $\Z\Delta(m)$ can then be listed in the form
$$[x(0)],\ \ldots,\ [x(r-1)],\ [t],\ \ldots,\ [t],$$
with $0\leq r\leq m$ and with $x(0)\leq\ldots\leq x(r-1)<t$. The canonical form for~$x$ depends on the parameters $t$, $r$ and $m-r$, for which we will use the following terminology.

\begin{definition} \label{D4.2}
Let $x$ be a morphism in $\Or(m,n)$. Then the \emph{terminus}
of~$x$, denoted $\terminus x$, is the integer~$t$ such that
$x[m]=[t]$; the \emph{rank} of~$x$, denoted $\rank x$, is the
number of zero-dimensional basis elements~$[a]$ in
$\Z\Delta(m)$ such that $x[a]\neq x[m]$; the \emph{corank}
of~$x$, denoted $\corank x$, is the number of zero-dimensional
basis elements~$[a]$ in $\Z\Delta(m)$ such that $a<m$ and
$x[a]=x[m]$.
\end{definition}

Thus the terminus of a morphism is a measure of the size of its image, the rank is a measure of its non-degeneracy, and the corank is a measure of its degeneracy.

Using Theorem~\ref{T4.1}, we immediately draw the following
conclusions.

\begin{proposition} \label{P4.3}
Let $x$ be a morphism in $\Or(m,n)$ with terminus~$t$, rank~$r$
and corank~$s$. Then $t$, $r$ and~$s$ are nonnegative integers
such that $t\leq n$ and $r+s=m$.
\end{proposition}

We will find the canonical forms by an induction on terminus; to be more precise, for $0\leq t\leq n$ we will express the morphisms in $\Or(-,n)$ with terminus at most~$t$ in in terms of the morphism 
$\d_{t+1}^{n-t}\iota_n$.
The set of morphisms of this kind is obviously isomorphic to $\Or(-,t)$, with $\d_{t+1}^{n-t}\iota_n$ corresponding to~$\iota_t$, so the induction is equivalent to an induction on~$n$, but it is easier to work inside the single graded set $\Or(-,n)$.

The passage to morphisms of terminus~$t$ from morphisms of
lower terminus is based on a cone construction~$\rho_t$: if $x$~is a morphism in $\Or(m,n)$ with $x[m]=[t']$ such that $t'<t$, then $\rho_t x$ is the obvious extension of~$x$ to a morphism in $\Or(m+1,n)$ with $(\rho_t x)[m+1]=[t]$. The precise definition is as follows.

\begin{definition} \label{D4.4}
Let $x$ be a morphism in $\Or(m,n)$ with terminus less
than~$t$. Then the \emph{$t$-cone} on~$x$ is the morphism
$\rho_t x$ in $\Or(m+1,n)$ given on basis elements as follows:
$(\rho_t x)[m+1]=[t]$; if
$$x[\mathbf{a}]=\sum_{\mathbf{b}}x_{\mathbf{a}\mathbf{b}}[\mathbf{b}]$$
then
$$(\rho_t x)[\mathbf{a}]=\sum_{\mathbf{b}}x_{\mathbf{a}\mathbf{b}}[\mathbf{b}],\quad
 (\rho_t x)[\mathbf{a},m+1]=\sum_{\mathbf{b}}x_{\mathbf{a}\mathbf{b}}[\mathbf{b},t].$$
\end{definition}

It is easy to see that $\rho_t x$ as defined here really is a
morphism in $\Or(m+1,n)$, with terminus~$t$, rank~$m$ and corank zero. It
is also easy to verify the following result.

\begin{proposition} \label{P4.5}
For $t>0$, let $S_{t-1}$ be the subset of $\Or(-,n)$ consisting of the
morphisms with terminus less than~$t$. Then $S_{t-1}$~is closed
under the face, degeneracy and wedge operations. The $t$-cone
construction
$$\rho_t\colon S_{t-1}\to\Or(-,n)$$
commutes with the operations in~$S_{t-1}$, and
$$\rho_t\d_t^{n-t+1}\iota_n=\d_{t+1}^{n-t}\iota_n.$$
\end{proposition}

The canonical form for a morphism~$x$ of rank~$r$ will be expressed in terms of simpler morphisms
$$\alpha_{r-1}x,\ldots,\alpha_0 x,\gamma x.$$
We will now describe these simpler morphisms, beginning with $\gamma x$.

Let $x$ be a morphism in $\Or(m,n)$ with terminus~$t$ and
rank~$r$; then $\gamma x$ is the morphism in $\Or(r,n)$ constructed as follows. For each basis element of the form $[\mathbf{a},r]$ with $[\mathbf{a}]$ a basis element in $\Z\Delta(r-1)$, we can write $x[\mathbf{a},r]$ in the form
$$x[\mathbf{a},r]
 =\sum_{\mathbf{b}}(x'_{\mathbf{a}\mathbf{b}}[\mathbf{b}]
 +x''_{\mathbf{a}\mathbf{b}}[\mathbf{b},t])$$
with the sum running over basis elements~$[\mathbf{b}]$ for
$\Z\Delta(t-1)$; then
$$\gamma x\colon\Z\Delta(r)\to\Z\Delta(n)$$
is to be the group homomorphism given on basis elements by $(\gamma x)[r]=[t]$ and by
$$(\gamma x)[\mathbf{a}]=\sum_{\mathbf{b}}x''_{\mathbf{a}\mathbf{b}}[\mathbf{b}],\quad
 (\gamma x)[\mathbf{a},r]=\sum_{\mathbf{b}}x''_{\mathbf{a}\mathbf{b}}[\mathbf{b},t].$$
Essentially, $\gamma x$ is obtained by selecting the terms of the form $[\mathbf{b},t]$ in the images $x[\mathbf{a},r]$.

We find that $\gamma x$ has the following properties.

\begin{proposition} \label{P4.6}
Let $x$ be a morphism in $\Or(-,n)$ with terminus~$t$ and
rank~$r$. Then $\gamma x$ is a morphism in $\Or(r,n)$ with
terminus~$t$, rank~$r$ and corank zero. If $r=0$ then $\gamma
x=\d_0^t\d_{t+1}^{n-t}\iota_n$\textup{;} if $r>0$ then $\gamma
x$ is a $t$-cone.
\end{proposition}

\begin{proof}
If $a=[a_0,\ldots,a_q]$ is a basis element for $\Z\Delta(m)$ with $a_q<r$, then it follows from Theorem~\ref{T4.1} that $xa$ is a sum of basis elements $[b_0,\ldots,b_q]$ with $b_q<t$. From this one can show that $\gamma x$ is a chain map. The rest follows easily.
\end{proof}

Now let $x$ be a morphism in $\Or(r+s,n)$ with terminus~$t$, rank~$r$ and corank~$s$, so that $\gamma x\in\Or(r,n)$. It is convenient to define a morphism $\beta x\in\Or(r+s,n)$ by the formula
$$\beta x=\e_r^s\gamma x,$$
so that $\beta x$ has the same domain as~$x$ and accounts for the same terms as $\gamma x$. For $0\leq p<r$ we then define chain maps
$$\alpha_p x\colon\Z\Delta(p+s+1)\to\Z\Delta(n)$$
by the formula
$$\alpha_p x
 =\d_{p+1}^{r-p-1}(x-\beta x)
 +\e_p\d_{p+1}^{r-p}\beta x.$$
Essentially, $\alpha_p x$ accounts for the terms $[b_0,\ldots,b_q]$ with $b_q<t$ in the images $x[a_0,\ldots,a_{q-1},r]$ for $a_{q-1}\leq p$, but in practice it is easier to work with $\alpha_p x$ by using algebraic formulae. There is one particularly simple case, as follows.

\begin{example} \label{E4.7}
Let $x$ be a morphism in $\Or(r,n)$ which is equal to
$\d_0^t\d_{t+1}^{n-t}\iota_n$ or is a $t$-cone. Then $\rank
x=r$, $\corank x=0$, $\beta x=\gamma x=x$, and
$$\alpha_p x=\e_p\d_{p+1}^{r-p}x\quad (0\leq p<r).$$
\end{example}

In general we get the following result.

\begin{proposition} \label{P4.8}
Let $x$ be a morphism in $\Or(-,n)$ with terminus~$t$, rank~$r$
and corank~$s$. If $s=0$ and $0\leq p<r$ then $\alpha_p x$ is a
morphism in $\Or(p+s+1,n)$ with terminus less than~$t$. If
$s>0$ and $0\leq p<r$ then $\alpha_p x$ is a morphism in
$\Or(p+s+1,n)$ with terminus~$t$, rank $p+2$ and corank $s-1$.
\end{proposition}

\begin{proof}
It is clear that $\alpha_p x$ is an augmentation-preserving
chain map from $\Z\Delta(p+s+1)$ to $\Z\Delta(n)$. To show that
$\alpha_p x$ is a morphism in~$\Or$, we must show that
$\alpha_p x$ takes each basis element~$a$ to a sum of basis
elements.

Suppose that $a$~does not have a term $p+1$. Then $(\alpha_p
x)a$ is a sum of basis elements because $\d_{p+1}\alpha_p x=\d_{p+1}^{r-p}x$ and because $x$~is a morphism in~$\Or$.

Suppose that $a$~has a term $p+1$ and a term greater than
$p+1$. Then $(\d_{p+1}^{r-p-1}\beta x)a=0$, because $\beta x$
has the form $\e_r^s\gamma x$, so that
$$(\alpha_p x)a=(\d_{p+1}^{r-p-1}x)a+(\e_p\d_{p+1}^{r-p}\beta x)a.$$
It now follows that
$(\alpha_p x)a$~is a sum of basis elements because $x$~and $\beta x$ are morphisms in~$\Or$.

Suppose that $a=[\mathbf{a},p+1]$. Then $(\d_{p+1}^{r-p-1}\beta
x)a$ is the sum of the terms of the form $[\mathbf{b},t]$ in
$(\d_{p+1}^{r-p-1}x)a$, so $\{\d_{p+1}^{r-p-1}(x-\beta x)\}a$
is the sum of the remaining terms in $(\d_{p+1}^{r-p-1}x)a$,
and it again follows that $(\alpha_p x)a$ is a sum of basis
elements.

This shows that $(\alpha_p x)a$ is a sum of basis elements in
all cases; therefore $\alpha_p x$ is a morphism in
$\Or(p+s+1,n)$.

We will now consider the terminus and corank of $\alpha_p x$.
For $p+2\leq i\leq p+s+1$ it follows from the calculations
above and from Proposition~\ref{P4.3} that
$$(\alpha_p x)[i]=(\d_{p+1}^{r-p-1}x)[i]=x[r+i-p-1]=[t];$$
on the other hand,
$$(\alpha_p x)[p+1]
 =(x-\beta x)[r]+(\beta x)[p]
 =[t]-[t]+(\beta x)[p]
 =(\gamma x)[p]
 \neq [t].$$
If $s=0$ it now follows that $\alpha_p x$ has terminus less
than~$t$; if $s>0$ it follows that $\alpha_p x$ has
terminus~$t$, rank $p+2$ and corank $s-1$.

This completes the proof.
\end{proof}

We will express a morphism~$x$ in terms of the morphisms
$\alpha_p x$ and $\gamma x$ by using the following notation.

\begin{notation} \label{N4.9}
In a set with operations $\d_i$~and~$\w_i$ we write
$$u({}\w_k v)=u\w_k v,$$
regarding $({}\w_k v)$ as an operator which acts on the right,
and for $l\geq 1$ we write
$$u\w_k^l v
 =u({}\w_k\d_{k+1}^{l-1}v)({}\w_k\d_{k+1}^{l-2}v)
 \ldots({}\w_k\d_{k+1}v)({}\w_k v).$$
We also write
$$(u\w_k^l{})v=u\w_k^l v,$$
regarding $(u\w_k^l{})$ as an operator which acts on the left,
and for $r\geq 0$ we write
$$\L^r(u_{r-1},\ldots,u_0,v)
 =\d_r(u_{r-1}\w_{r-1}^1{})\d_{r-1}(u_{r-2}\w_{r-2}^2{})\ldots
 \d_1(u_0\w_0^r{})v.$$

Note in particular that $\L^0(v)=v$.
\end{notation}

In $\Or(-,n)$ an induction based on Definition~\ref{D3.5} gives
the following result.

\begin{proposition} \label{P4.10}
If $u$~and~$v$ are morphisms in $\Or(-,n)$ and if $\d_k
u=\d_{k+1}^l v$, then $u\w_k^l v$ is defined and
$$u\w_k^l v=\e_{k+1}^l u-\e_k^{l+1}\d_{k+1}^l v+\e_k v.$$
\end{proposition}

For an iterated wedge $z=u\w_k^l v$ in $\Or(-,n)$ as in Proposition~\ref{P4.10}, one finds that $\d_{k+2}^l z=u$, $\d_k z=v$, and
$$\d_{k+1}z=\e_{k+1}^{l-1}u-\e_k\d_{k+1}^l v+v,$$
so that $\d_{k+1}(u\w_k^l)$ is a pasting composite of $u$~and~$v$ and so that $u\w_k^l v$ joins this composite to its factors. The operation~$\L^r$ is an iterated pasting composite.

The desired canonical form is now as follows.

\begin{theorem} \label{T4.11}
Let $x$ be a morphism in $\Or(-,n)$, with rank~$r$ and 
corank~$s$. Then
$$x=\L^r(\alpha_{r-1}x,\ldots,\alpha_0 x,\e_r^s\gamma x).$$
\end{theorem}

\begin{proof}
Recall that
$$\alpha_p x
 =\d_{p+1}^{r-p-1}(x-\beta x)+\e_p\d_{p+1}^{r-p}\beta x,$$
where $\beta x=\e_r^s\gamma x$. For $0\leq p\leq r$, let
$$v_p=\e_p^{r-p}\d_p^{r-p}(x-\beta x)+\beta x,$$
so that, in particular, $v_r=x$. It suffices to show that
\begin{align*}
 &v_0=\e_r^s\gamma x,\\
 &v_{p+1}=\d_{p+1}(\alpha_p x\w_p^{r-p}v_p)\quad (0\leq p<r);
\end{align*}
we proceed as follows.

The morphisms $\d_0^r x$ and $\d_0^r\beta x$ are morphisms in
$\Or(s,n)$ such that
$$(\d_0^r x)[i]=(\d_0^r\beta x)[i]=[t]$$
for $0\leq i\leq s$. By Theorem~\ref{T4.1}, if $a$~is a basis element for $\Z\Delta(s)$ with positive dimension then
$$(\d_0^r x)a=(\d_0^r\beta x)a=0;$$
therefore $\d_0^r
x=\d_0^r\beta x$. It follows that $v_0=\beta x=\e_r^s\gamma x$.

For $0\leq p<r$ it is straightforward to verify that
$\d_p\alpha_p x=\d_{p+1}^{r-p}v_p$ and that
$$v_{p+1}=\d_{p+1}
 (\e_{p+1}^{r-p}\alpha_p x-\e_p^{r-p+1}\d_{p+1}^{r-p}v_p+\e_p v_p);$$
therefore $v_{p+1}=\d_{p+1}(\alpha_p x\w_p^{r-p}v_p)$.

This completes the proof.
\end{proof}

We can now prove the main result.

\begin{theorem} \label{T4.12}
Every morphism in $\Or(-,n)$ can be expressed in terms
of~$\iota_n$ by using the face, degeneracy and wedge
operations.
\end{theorem}

\begin{proof}
Let $S_t$ be the set of morphisms with terminus at most~$t$. We
will prove by induction on~$t$ that $S_t$~is generated by
$\d_{t+1}^{n-t}\iota_n$; the case $t=n$ will then give the
result.

Let $x$ be a morphism in $S_t\setminus S_{t-1}$, so that $x$~has terminus~$t$. By Theorem~\ref{T4.11}, $x$~has an expression in terms of the morphisms $\gamma x$ and $\alpha_p x$. By Proposition~\ref{P4.6}, $\gamma x=\d_0^t\d_{t+1}^{n-t}\iota_n$ or $\gamma x$ is a $t$-cone. By Proposition~\ref{P4.8}, $\alpha_p x\in S_t$, and if $\alpha_p x\in S_t\setminus S_{t-1}$ then $\alpha_p x$ has lower corank than~$x$. By a recursion on corank, it follows that $S_t$~has a set of generators
consisting of the morphism $\d_0^t\d_{t+1}^{n-t}\iota_n$, the
$t$-cones, and the morphisms in~$S_{t-1}$. It therefore suffices to show that $t$-cones and members of~$S_{t-1}$ can be expressed in terms of $\d_{t+1}^{n-t}\iota_n$; equivalently, it suffices to show that $\rho_t y$ and~$y$ can be expressed in terms of $\d_{t+1}^{n-t}\iota_n$ for every morphism~$y$ in~$S_{t-1}$.

For $t=0$ there is nothing to prove, because $S_{-1}$~is empty.

Now suppose that $t>0$. By the inductive hypothesis, $y$~can be expressed in terms of $\d_t^{n-t+1}\iota_n$. It now follows from Proposition~\ref{P4.5} that $\rho_t y$ can be
expressed in terms of $\d_{t+1}^{n-t}\iota_n$. It also follows that $y$~itself 
can be expressed in terms of $\d_{t+1}^{n-t}\iota_n$, because
$\d_t^{n-t+1}\iota_n=\d_t\d_{t+1}^{n-t}\iota_n$.

This completes the proof.
\end{proof}

\section{Further properties of the nerves of orientals} \label{S5}

We have shown in Theorem~\ref{T4.12} that $\Or(-,n)$ is
generated by~$\iota_n$, and in Theorem~\ref{T4.11} we have found canonical forms for the members of $\Or(-,n)$. In the later sections of this paper we
will find identities determining the structure of $\Or(-,n)$
completely. In the present section we give some auxiliary results on the canonical forms, for later use.

First we describe the canonical forms for certain faces.

\begin{proposition} \label{P5.1}
Let $x$ be a morphism in $\Or(-,n)$ with rank~$r$, and let
$\d_i x$ be a face with $0\leq i<r$. Then $\d_i x$ has rank $r-1$ and corank~$s$, and
\begin{align*}
 &\gamma\d_i x=\d_i\gamma x,\\
 &\alpha_p\d_i x=\alpha_p x\quad (0\leq p<i),\\
 &\alpha_p\d_i x=\d_i\alpha_{p+1}x\quad (i\leq p<r-1).
\end{align*}
\end{proposition}

\begin{proof}
Let the terminus of~$x$ be~$t$ and the corank be~$s$. It
follows from Proposition~\ref{P4.3} that $\d_i x$ has
terminus~$t$, rank $r-1$ and corank~$s$. It is clear from the
construction that $\gamma\d_i x=\d_i\gamma x$, it follows that
$$\beta\d_i x
 =\e_{r-1}^s\gamma\d_i x
 =\e_{r-1}^s\d_i\gamma x
 =\d_i\e_r^s\gamma x
 =\d_i\beta x,$$
and for $0\leq p<r-1$ it follows that
$$\alpha_p\d_i x
 =\d_{p+1}^{r-p-2}(\d_i x-\d_i\beta x)+\e_p\d_{p+1}^{r-p-1}\d_i\beta x.$$
If now $0\leq p<i$ then
$$\alpha_p\d_i x
 =\d_{p+1}^{r-p-1}(x-\beta x)+\e_p\d_{p+1}^{r-p}\beta x
 =\alpha_p x;$$
if $i\leq p<r-1$ then
$$\alpha_p\d_i x
 =\d_i\d_{p+2}^{r-p-2}(x-\beta x)+\d_i\e_{p+1}\d_{p+2}^{r-p-1}\beta x
 =\d_i\alpha_{p+1}x.$$
\end{proof}

Next we consider wedges. If $x$~is a wedge, say $x=u\w_i v$, then $u=\d_{i+2}x$ and $v=\d_i x$ by Proposition~\ref{P3.6}. A wedge operation~$\w_i$ is therefore entirely determined by its image set, which will be denoted $\im\w_i$. We will now describe the canonical forms of wedges in terms of these image sets.

\begin{proposition} \label{P5.2}
Let $x$ be a morphism of rank~$r$ in $\Or(-,n)$.

If $x\in\im\w_i$ with $i\leq r-3$ then $\gamma x\in\im\w_i$,
and $\alpha_p x\in\im\w_i$ for $i+2\leq p<r$.

If $x\in\im\w_{r-2}$ then $\gamma x\in\im\e_{r-2}$ and
$\alpha_{r-1}x\in\im\w_{r-2}$.

If $x\in\im\w_{r-1}$ then $\alpha_{r-1}x\in\im\w_{r-1}$.

If $x\in\im\w_i$ with $r\leq i$ then $\alpha_p
x\in\im\w_{i-r+p+1}$ for $0\leq p<r$.
\end{proposition}

\begin{proof}
Let the corank of~$x$ be~$s$, so that
$$\alpha_p x
 =\d_{p+1}^{r-p-1}x
 -\d_{p+1}^{r-p-1}\e_r^s\gamma x
 +\e_p\d_{p+1}^{r-p}\e_r^s\gamma x,$$
and recall from Proposition~\ref{P3.8} that a morphism is in
$\im\w_i$ if and only if it annihilates all basis elements
including $i,i+1,i+2$.

Suppose that $x\in\im\w_i$ with $i\leq r-3$. Then $\gamma x$
annihilates basis elements including $i,i+1,i+2$ because
$x$~annihilates basis elements including $i,i+1,i+2,r$;
therefore $\gamma x\in\im\w_i$. For $i+2\leq p<r$ it follows
that $\d_{p+1}^{r-p-1}x$, $\d_{p+1}^{r-p-1}\e_r^s\gamma x$ and
$\e_p\d_{p+1}^{r-p}\e_r^s\gamma x$ annihilate all basis
elements including $i,i+1,i+2$, and it then follows that
$\alpha_p x\in\im\w_i$.

Suppose that $x\in\im\w_{r-2}$. In this case $\gamma x$
annihilates all basis elements including $r-2,r-1$ because
$x$~annihilates all basis elements including $r-2,r-1,r$;
hence, by Proposition~\ref{P3.4}, $\gamma x\in\im\e_{r-2}$. As
in the previous case, $\alpha_{r-1}x\in\im\w_{r-2}$.

Suppose that $x\in\im\w_{r-1}$. Then $s\geq 1$, so that
$\e_r^s\gamma x$ and $\e_{r-1}\d_r\e_r^s\gamma x$ are in
$\im\w_{r-1}$. As before, it follows that
$\alpha_{r-1}x\in\im\w_{r-1}$.

Finally, suppose that $x\in\im\w_i$ with $r\leq i$, and suppose
that $0\leq p<r$. Then $i\leq r+s-2$, so $\d_{p+1}^{r-p-1}x$,
$\d_{p+1}^{r-p-1}\e_r^s\gamma x$ and
$\e_p\d_{p+1}^{r-p}\e_r^s\gamma x$ are all in
$\im\w_{i-r+p+1}$; therefore $\alpha_p x\in\im\w_{i-r+p+1}$.
\end{proof}

To finish this section, we describe the termini and coranks that appear in the evaluation of an expression of the form
$\L^r(u_{r-1},\ldots,u_0,v)$ (see Notation~\ref{N4.9}).

\begin{proposition} \label{P5.3}
Let $\L^r(u_{r-1},\ldots,u_0,v)$ be an expression defined in
$\Or(-,n)$, let $v_0=v$, and for $1\leq p\leq r$ let
$$w_p=u_{p-1}\w_{p-1}^{r-p+1}v_{p-1},\quad v_p=\d_p w_p,$$
so that $\L^r(u_{r-1},\ldots,u_0,v)=v_r$. Then the morphisms $v_p$~and~$w_p$ all have the same terminus
as~$v$. If $\rank v\geq r$ then the morphisms $v_p$~and~$w_p$
all have the same corank as~$v$.
\end{proposition}

\begin{proof}
Let $v\in\Or(m,n)$, so that $m\geq r$, and so that $v_p\in\Or(m,n)$ and $w_p\in\Or(m+1,n)$ for all relevant~$p$. Since $v_p=\d_p w_p$ by construction, and since $\d_{p-1}w_p=v_{p-1}$ by
Proposition~\ref{P3.6}, for $r-1\leq i\leq m$ we get
$$v_r[i]
 =w_r[i+1]
 =v_{r-1}[i]
 =\ldots
 =w_1[i+1]
 =v_0[i]
 =v[i].$$

Let the terminus of~$v$ be~$t$, so that $v[m]=t$; then $v_p[m]=t$ and $w_p[m+1]=t$, so the morphisms $v_p$~and~$w_p$ also have terminus~$t$. 

Now suppose that $\rank v\geq r$. Let the corank of~$v$ be~$s$, so that $m-s\geq r$. By the definition of corank, we have $v[i]=[t]$ if and only if $m-s\leq i\leq m$. It follows that $v_p[i]=[t]$ if and only if $m-s\leq i\leq m$ and that $w_p[j]=[t]$ if and only if $m+1-s\leq j\leq m+1$. Therefore $v_p$~and~$w_p$ also have corank~$s$.
\end{proof}

\section{Complicial identities} \label{S6}

In this section we define sets with complicial identities; they
will be simplicial sets with wedge operations subject to
certain axioms. We will then show that the simplicial sets
$\Or(-,n)$ satisfy these axioms.

\begin{definition} \label{D6.1}
A \emph{set with complicial identities} is a simplicial
set~$X$, together with \emph{wedges}
$$x\w_i y\in X_{m+1},$$
defined when $x,y\in X_m$ and $\d_i x=\d_{i+1}y$, such that the
following axioms hold.

\textup{(1)} If $x\w_i y$ is defined with $x,y\in X_m$, then
\begin{align*}
 &\d_j(x\w_i y)=\d_j x\w_{i-1}\d_j y\quad (0\leq j<i),\\
 &\d_i(x\w_i y)=y,\\
 &\d_{i+2}(x\w_i y)=x,\\
 &\d_j(x\w_i y)=\d_{j-1}x\w_i\d_{j-1}y\quad (i+3\leq j\leq m+1).
\end{align*}

\textup{(2)} If $x\in X_m$ and $0\leq i<m$ then
$$\e_i x=\e_i\d_{i+1}x\w_i x,\quad \e_{i+1}x=x\w_i\e_i\d_i x.$$

\textup{(3)} If $A$~is of the form $b\w_i(y\w_i z)$ then
$$A=(\d_{i+2}b\w_i y)\w_{i+1}\d_{i+1}A.$$

\textup{(4)} If $A$~is of the form $(x\w_i y)\w_{i+1}c$ then
$$A=\d_{i+2}A\w_i(y\w_i\d_i c).$$

\textup{(5)} The equality
$$[x\w_i\d_{i+1}(y\w_i z)]\w_i(y\w_i z)
 =(x\w_i y)\w_{i+1}[\d_{i+1}(x\w_i y)\w_i z]$$
holds whenever either side is defined.

\textup{(6)} If $A$~is of the form
$\d_{i+2}[(x\w_{i+1}y)\w_{i+1}(y\w_i z)]$, then the
equality
$$A\w_i(w\w_{i+1}\d_i A)=(\d_{i+3}A\w_i w)\w_{i+2}A$$
holds whenever either side is defined.

\textup{(7)} If $i\leq j-3$ then the equality
$$(x\w_i y)\w_j(z\w_i w)=(x\w_{j-1}z)\w_i(y\w_{j-1}w)$$
holds whenever either side is defined.

A \emph{morphism of sets with complicial identities} $f\colon
X\to Y$ is a sequence of functions $f\colon X_m\to Y_m$
commuting with the face, degeneracy and wedge operations.
\end{definition}

\begin{remark} \label{R6.2}
There is some redundancy in these axioms. Indeed, the
degeneracy operations on elements of positive dimension are
redundant because of axiom~(2). For similar reasons, one can
omit some of the simplicial identities, retaining only
$\d_i\d_j=\d_{j-1}\d_i$ (for $i<j$) and
$\d_i\e_i=\d_{i+1}\e_i=\id$.
\end{remark}

\begin{remark} \label{R6.3}
Axiom~(3) could be written simply as an equality
$$b\w_i(y\w_i z)
 =(\d_{i+2}b\w_i y)\w_{i+1}\d_{i+1}[b\w_i(y\w_i z)],$$
required when either side is defined. A similar remark applies
to axioms (4) and~(6).
\end{remark}

\begin{remark} \label{R6.4}
Note that axiom~(1) does not give a formula for $\d_{i+1}(x\w_i
y)$. In a sense, the operation~$\w_i$ exists in order to
construct the operation
$$(x,y)\mapsto\d_{i+1}(x\w_i y).$$
By using axioms (1), (2) and~(5), one can show that this binary
operation is associative and that it makes the $m$-dimensional
elements into the morphisms of a category. The objects are the
$(m-1)$-dimensional elements, the source and target of a
morphism~$x$ are $\d_i x$ and $\d_{i+1}x$, and the identity of
an object~$a$ is $\e_i a$.
\end{remark}

\begin{remark} \label{R6.5}
I intend to show in a future paper that sets with complicial
identities are equivalent to the complicial sets used by Verity
in~\cite{V}.
\end{remark}

We will now show that the axioms of Definition~\ref{D6.1} apply
to orientals.

\begin{proposition} \label{P6.6}
The graded sets $\Or(-,n)$ are sets with complicial identities.
\end{proposition}

\begin{proof}
From Section~\ref{S3} we know that $\Or(-,n)$ is a simplicial
set and that it has wedge operations with the correct domains
and codomains.

It is straightforward to verify axiom~(1).

Next we verify axiom~(5). The existence of the expression on
one side is equivalent to the existence of the expression on
the other, because the existence of either expression is
equivalent to the truth of the equalities
$$\d_i x=\d_{i+1}y,\quad \d_i y=\d_{i+1}z.$$
Also, if the two expression do exist, they have the same value,
namely
$$\e_{i+3}\e_{i+2}x-\e_i^3\d_{i+1}y
 +\e_{i+3}\e_i y-\e_i^3\d_{i+1}z+\e_{i+1}\e_i z.$$

The remaining axioms all say that expressions of certain forms
are equal to wedges. They can be proved by using
Proposition~\ref{P3.8}, which says that an element of
$\Or(-,n)$ is in the image of~$\w_i$ if and only if it is in
$\im\e_i+\im\e_{i+1}$. For example, suppose that the expression
on the left of axiom~(6) is defined, and let
$$B=A\w_i(w\w_{i+1}\d_i A).$$
We find that $B\in\im\e_{i+2}+\im\e_{i+3}$; therefore $B$~has
the form $B'\w_{i+2}B''$. Using axiom~(1) we then find that
\begin{align*}
 &B'=\d_{i+4}B=\d_{i+3}A\w_i\d_{i+3}(w\w_{i+1}\d_i A)=\d_{i+3}A\w_i w,\\
 &B''=\d_{i+2}B=A;
\end{align*}
therefore $B=(\d_{i+3}A\w_i w)\w_{i+2}A$.

This completes the proof.
\end{proof}

\section{Consequences of the complicial identities} \label{S7}

We have shown in Theorem~\ref{T4.12} that $\Or(-,n)$ is
generated by the identity morphism~$\iota_n$. In
Section~\ref{S8} we will show that $\Or(-,n)$ is freely
generated by~$\iota_n$ subject to the complicial identities of
Definition~\ref{D6.1}; in other words, we will show that the
identities determine the entire structure. In order to do this,
we now consider consequences of the identities in general.

There are three parts to this section. In the first part (Propositions \ref{P7.1}--\ref{P7.3}) we find the domains of definition and most of the faces for the expressions $u\w_k^l v$ and $\L^r(u_{r-1},\ldots,u_0,v)$ of Notation~\ref{N4.9}. In the second part (Proposition~\ref{P7.4}) we find a sufficient condition for an expression $\L^r(u_{r-1},\ldots,u_0,v)$ to collapse to~$v$. The third part, concluding with Proposition~\ref{P7.11}, is designed to find conditions implying that an expression of the form $\L^r(u_{r-1},\ldots,u_0,\e_r^s w)$ is a wedge.

First we find the domains of definition and most of the faces for the expressions $u\w_k^l v$.

\begin{proposition} \label{P7.1}
An expression $u\w_k^l v$ in a set with complicial identities
exists if and only if $\d_k u=\d_{k+1}^l v$. If the expression
does exist, then
\begin{align*}
 &\d_i(u\w_k^l v)=\d_i u\w_{k-1}^l\d_i v\quad (i<k),\\
 &\d_k(u\w_k^l v)=v,\\
 &\d_{k+2}(u\w_k^1 v)=u,\\
 &\d_{i+1}(u\w_k^l v)=u\w_k^{l-1}\d_i v\quad
  (l>1,\ k<i\leq k+l),\\
 &\d_{i+1}(u\w_k^l v)=\d_{i-l+1}u\w_k^l\d_i v\quad (k+l<i).
\end{align*}
\end{proposition}

\begin{proof}
Recall from Notation~\ref{N4.9} that $u\w_k^1 v=u\w_k v$ and that
$$u\w_k^l v=(u\w_k^{l-1}v)\w_k v$$
for $l>1$. Recall also from Definition~\ref{D6.1} that a wedge $x\w_k y$ is defined if and only if $\d_k x=\d_{k+1}y$, and that $\d_k(x\w_k y)=y$ whenever $x\w_k y$ is defined. It follows from this that $u\w_k^l v$ is defined if and only if $\d_k
u=\d_{k+1}^l v$. The formulae for $\d_i(u\w_k^l v)$ follow by induction on~$l$ from Definition~\ref{D6.1}(1).
\end{proof}

The domains of definition for the expressions $\L^r(u_{r-1},\ldots,u_0,v)$ are as follows.

\begin{proposition} \label{P7.2}
An expression $\L^r(u_{r-1},\ldots,u_0,v)$ exists in a set with
complicial identities if and only
$$\d_p u_p=\L^p(u_{p-1},\ldots,u_0,\d_{p+1}^{r-p}v)$$
for $0\leq p<r$.
\end{proposition}

\begin{proof}
For $0\leq p\leq r$ let
$$v_p
 =\d_p(u_{p-1}\w_{p-1}^{r-p+1}{})\d_{p-1}(u_{p-2}\w_{p-2}^{r-p+2}{})\ldots
 \d_1(u_0\w_0^r{})v,$$
so that $\L^r(u_{r-1},\ldots,u_0,v)=v_r$. For $0\leq p<r$, by
Proposition~\ref{P7.1}, $v_{p+1}$~exists if and only if $\d_p
u_p=\d_{p+1}^{r-p}v_p$; it therefore suffices to show that
$$\d_{p+1}^{r-p}v_p=\L^p(u_{p-1},\ldots,u_0,\d_{p+1}^{r-p}v).$$
But for $0<q\leq p$ it follows from Proposition~\ref{P7.1} that
$$\d_{p+1}^{r-p}\d_q(u_{q-1}\w_{q-1}^{r-q+1}{})
 =\d_q\d_{p+2}^{r-p}(u_{q-1}\w_{q-1}^{r-q+1}{})
 =\d_q(u_{q-1}\w_{q-1}^{p-q+1}{})\d_{p+1}^{r-p},$$
and this gives the result.
\end{proof}

The faces of the expressions $\L^r(u_{r-1},\ldots,u_0,v)$ are as follows.

\begin{proposition} \label{P7.3}
In a set with complicial identities the following equalities
are valid whenever their left sides are defined:
\begin{align*}
 &\d_i\L^r(u_{r-1},\ldots,u_0,v)
 =\L^{r-1}
 (\d_i u_{r-1},\ldots,\d_i u_{i+1},u_{i-1},\ldots,u_0,\d_i v)\quad
 (i<r),\\
 &\d_0\L^0(v)=\d_0 v,\\
 &\d_r\L^r(u_{r-1},\ldots,u_0,v)=\d_r u_{r-1}\quad (r>0),\\
 &\d_i\L^r(u_{r-1},\ldots,u_0,v)
  =\L^r(\d_i u_{r-1},\d_{i-1}u_{r-2},\ldots,\d_{i-r+1}u_0,\d_i v)\quad
  (i>r).
\end{align*}
\end{proposition}

\begin{proof}
Suppose first that $i<r$. Using Proposition~\ref{P7.1}, for
$r\geq p>i+1$ we get
$$\d_i\d_p(u_{p-1}\w_{p-1}^{r-p+1}{})
 =\d_{p-1}\d_i(u_{p-1}\w_{p-1}^{r-p+1}{})
 =\d_{p-1}(\d_i u_{p-1}\w_{p-2}^{r-p+1}{})\d_i,$$
we then get
$$\d_i\d_{i+1}(u_i\w_i^{r-i}{})=\d_i\d_i(u_i\w_i^{r-i}{})=\d_i,$$
and for $i\geq p>0$ we get
$$\d_i\d_p(u_{p-1}\w_{p-1}^{r-p+1}{})
 =\d_p\d_{i+1}(u_{p-1}\w_{p-1}^{r-p+1}{})
 =\d_p(u_{p-1}\w_{p-1}^{r-p}{})\d_i;$$
therefore
$$\d_i\L^r(u_{r-1},\ldots,u_0,v)
 =\L^{r-1}
 (\d_i u_{r-1},\ldots,\d_i u_{i+1},u_{i-1},\ldots,u_0,\d_i v).$$

It is obvious that $\d_0\L^0(v)=\d_0 v$, because $\L^0(v)=v$.

For $r>0$ the equality $\d_r\L^r(u_{r-1},\ldots,u_0,v)=\d_r
u_{r-1}$ is immediate from Proposition~\ref{P7.1}.

Finally, suppose that $i>r$. For $r\geq p>0$ it follows from
Proposition~\ref{P7.1} that
$$\d_i\d_p(u_{p-1}\w_{p-1}^{r-p+1}{})
 =\d_p\d_{i+1}(u_{p-1}\w_{p-1}^{r-p+1}{})
 =\d_p(\d_{i-r+p}u_{p-1}\w_{p-1}^{r-p+1}{})\d_i;$$
therefore
$$\d_i\L^r(u_{r-1},\ldots,u_0,v)
  =\L^r(\d_i u_{r-1},\d_{i-1}u_{r-2},\ldots,\d_{i-r+1}u_0,\d_i v).$$
\end{proof}

Next we give the collapsing result.

\begin{proposition} \label{P7.4}
If $x$~is an element of dimension at least~$r$ in a set with
complicial identities, then
$$\L^r(\e_{r-1}\d_r x,\ \e_{r-2}\d_{r-1}^2 x,\ \ldots,\ \e_0\d_1^r x,\ x)=x.$$
\end{proposition}

\begin{proof}
Repeated applications of Definition~\ref{D6.1}(2) show that for
$0\leq p<r$ we have
$$\d_{p+1}(\e_p\d_{p+1}^{r-p}x\w_p^{r-p}x)=\d_{p+1}\e_p x=x.$$
The result follows.
\end{proof}

The rest of this section is designed to prove Proposition~\ref{P7.11}, giving conditions for an expression of the form $\L^r(u_{r-1},\ldots,u_0,\e_r^s w)$ to be a wedge. We use a sequence of lemmas showing that various constituent expressions are wedges.

\begin{lemma} \label{L7.5}
Let $u\w_k v$ be a wedge in a set with complicial identities.
If $i=k-1$ or $i=k$ and if $u,v\in\im\w_i$ then $\d_{k+1}(u\w_k
v)\in\im\w_i$.
\end{lemma}

\begin{proof}
Suppose that $u=u'\w_{k-1}u''$ and $v=v'\w_{k-1}v''$. The
existence of the wedge $u\w_k v$ implies that
$v'=\d_{k+1}v=\d_k u=\d_k(u'\w_{k-1}u'')$; therefore
$$u\w_k v
 =(u'\w_{k-1}u'')\w_k[\d_k(u'\w_{k-1}u'')\w_{k-1}v''].$$
It now follows from Definition~\ref{D6.1}(5) that
$$\d_{k+1}(u\w_k v)=u'\w_{k-1}\d_k(u''\w_{k-1}v'').$$

If $u=u'\w_k u''$ and $v=v'\w_k v''$, then it similarly follows
that $u''=\d_{k+1}(v'\w_k v'')$ and that
$$\d_{k+1}(u\w_k v)=\d_{k+1}(u'\w_k v')\w_k v''.$$
\end{proof}

\begin{lemma} \label{L7.6}
In a set with complicial identities, let $A$ be an element of
the form $x\w_{i+1}y$ or $y\w_i z$ or
$\d_{i+2}[(x\w_{i+1}y)\w_{i+1}(y'\w_i z)]$. Then elements of
the form $A\w_i(w\w_{i+1}v)$ are in $\im\w_{i+2}$, and elements
of the form $(u\w_i w)\w_{i+2}A$ are in $\im\w_i$.
\end{lemma}

\begin{proof}
Note first that $A$~must be $m$-dimensional with $0\leq i\leq
m-3$, because of the existence of $A\w_i(w\w_{i+1}v)$ or of
$(u\w_i w)\w_{i+2}A$.

Suppose that $B$~has the form $A\w_i(w\w_{i+1}v)$; then $\d_i
A=\d_{i+1}(w\w_{i+1}v)=v$, so
$$B=A\w_i(w\w_{i+1}\d_i A).$$
Because of Definition~\ref{D6.1}(6), to show that $B\in\im\w_{i+2}$ it suffices to show that
$A$~is of the form
$$\d_{i+2}[(x\w_{i+1}y)\w_{i+1}(y\w_i z)],$$
and we do this as follows: if $A=x\w_{i+1}y$ then
\begin{align*}
 A
 &=\d_{i+2}\e_{i+2}A\\
 &=\d_{i+2}(A\w_{i+1}\e_{i+1}\d_{i+1}A)\\
 &=\d_{i+2}(A\w_{i+1}\e_{i+1}y)\\
 &=\d_{i+2}[(x\w_{i+1}y)\w_{i+1}(y\w_i\e_i\d_i y)];
\end{align*}
if $A=y\w_i z$ then similarly
$$A=\d_{i+2}[(\e_{i+1}\d_{i+2}y\w_{i+1}y)\w_{i+1}(y\w_i z)];$$
if $A=\d_{i+2}[(x\w_{i+1}y)\w_{i+1}(y'\w_i z)]$ then
$y=\d_{i+1}(x\w_{i+1}y)=\d_{i+2}(y'\w_i z)=y'$, so
$$A=\d_{i+2}[(x\w_{i+1}y)\w_{i+1}(y\w_i z)].$$

A similar argument shows that elements of the form $(u\w_i w)\w_{i+2}A$ are in $\im\w_i$.
\end{proof}

\begin{lemma} \label{L7.7}
In a set with complicial identities, if $r\geq i+3$ then
$$\e_r(x\w_i y)=\e_{r-1}x\w_i\e_{r-1}y$$
whenever $x\w_i y$ is defined.
\end{lemma}

\begin{proof}
The proof is by induction on~$r$. We use
Definition~\ref{D6.1}(1) and~(2) repeatedly.

In the case $r=i+3$ we have
\begin{align*}
 \e_{i+3}(x\w_i y)
 &=(x\w_i y)\w_{i+2}\e_{i+2}\d_{i+2}(x\w_i y)\\
 &=(x\w_i y)\w_{i+2}\e_{i+2}x\\
 &=(x\w_i y)\w_{i+2}(x\w_{i+1}\e_{i+1}\d_{i+1}x).
\end{align*}
It follows from Lemma~\ref{L7.6} that $\e_{i+3}(x\w_i y)$
is in $\im\w_i$, and it then follows that
\begin{align*}
 \e_{i+3}(x\w_i y)
 &=\d_{i+2}\e_{i+3}(x\w_i y)\w_i\d_i\e_{i+3}(x\w_i y)\\
 &=\e_{i+2}\d_{i+2}(x\w_i y)\w_i\e_{i+2}\d_i(x\w_i y)\\
 &=\e_{i+2}x\w_i\e_{i+2}y.
\end{align*}

For $r>i+3$ it follows from the inductive hypothesis and
Definition~\ref{D6.1}(7) that
\begin{align*}
 \e_r(x\w_i y)
 &=(x\w_i y)\w_{r-1}\e_{r-1}\d_{r-1}(x\w_i y)\\
 &=(x\w_i y)\w_{r-1}\e_{r-1}(\d_{r-2}x\w_i\d_{r-2}y)\\
 &=(x\w_i y)\w_{r-1}(\e_{r-2}\d_{r-2}x\w_i\e_{r-2}\d_{r-2}y)\\
 &=(x\w_{r-2}\e_{r-2}\d_{r-2}x)\w_i(y\w_{r-2}\e_{r-2}\d_{r-2}y)\\
 &=\e_{r-1}x\w_i\e_{r-1}y.
\end{align*}

This completes the proof.
\end{proof}

\begin{lemma} \label{L7.8}
In a set with complicial identities, an element of the form
$$(u\w_i u')\w_k(v\w_i v')$$
is in $\im\w_i$ for $i\leq k$ and is in $\im\w_{i+1}$ for
$i\geq k-1$.
\end{lemma}

\begin{proof}
For $i=k-1$ and for $i=k$ the results are trivial or are
contained in Definition \ref{D6.1}(3) and~(4); for $i=k-2$ and
for $i=k+1$ the results follow from Lemma~\ref{L7.6}; for
$i\leq k-3$ and for $i\geq k+2$ the results are contained in
Definition~\ref{D6.1}(7).
\end{proof}

\begin{lemma} \label{L7.9}
In a set with complicial identities, an element of the form
$$\d_{i+3}[u\w_{i+2}^l\d_{i+2}(u'\w_{i+1}^{l+1}v)]$$
is in $\im\w_i$ if $u$~and~$v$ are in $\im\w_i$.
\end{lemma}

\begin{proof}
Let
$$z=u\w_{i+2}^l\d_{i+2}(u'\w_{i+1}^{l+1}v).$$
Because of Definition~\ref{D6.1}(1), it suffices to show that
$z\in\im\w_i$. We will do this by induction on~$l$.

Let
$$A=\d_{i+2}(u'\w_{i+1}^{l+1}v)$$
and let
$$u''=\begin{cases}
 u& (l=1),\\
 u\w_{i+2}^{l-1}\d_{i+3}A& (l>1),
 \end{cases}$$
so that
$$z=u''\w_{i+2}A.$$
Note that
$$A=\d_{i+2}[(u'\w_{i+1}^l\d_{i+2}v)\w_{i+1}v]$$
with $u'\w_{i+1}^l\d_{i+2}v\in\im\w_{i+1}$ by the definition
of~$\w_{i+1}^l$ and with $v\in\im\w_i$ by hypothesis. By
Lemma~\ref{L7.6}, to show that $z\in\im\w_i$ it suffices
to show that $u''\in\im\w_i$.

Suppose that $l=1$. Then $u''=u$, so $u''\in\im\w_i$ by
hypothesis.

Suppose that $l>1$. Using Proposition~\ref{P7.1} we get
\begin{align*}
 u''
 &=u\w_{i+2}^{l-1}\d_{i+3}A\\
 &=u\w_{i+2}^{l-1}\d_{i+2}\d_{i+4}(u'\w_{i+1}^{l+1}v)\\
 &=u\w_{i+2}^{l-1}\d_{i+2}(u'\w_{i+1}^l\d_{i+3}v).
\end{align*}
Since $v\in\im\w_i$, it follows from Definition~\ref{D6.1}(1)
that $\d_{i+3}v\in\im\w_i$, and it then follows from the
inductive hypothesis that $u''\in\im\w_i$.

This completes the proof.
\end{proof}

\begin{lemma} \label{L7.10}
In a set with complicial identities, let $v'$ be an element of
the form $\d_{k+1}(u\w_k^l v)$ with $v\in\im\w_i$. Then
$v'\in\im\w_i$ whenever one of the following sets of conditions
is satisfied:

\textup{(1)} $i\leq k-2$ and $u\in\im\w_i$\textup{;}

\textup{(2)} $k\leq i<k+l$ and $l\geq 2$\textup{;}

\textup{(3)} $k+l\leq i$ and $u\in\im\w_{i-l+1}$.
\end{lemma}

\begin{proof}
Let $u_0=u$, and for $0<j\leq l$ let
$$u_j=u_{j-1}\w_k\d_{k+1}^{l-j}v;$$
we must show in each case that $\d_{k+1}u_l\in\im\w_i$.

(1) Since $v\in\im\w_i$ with $i\leq k-2$, it follows from
Definition~\ref{D6.1}(1) that $\d_{k+1}^{l-j}v\in\im\w_i$ for
$0\leq j<l$. Since $u_0\in\im\w_i$, it follows from
Lemma~\ref{L7.8} that $u_1,\ldots,u_l\in\im\w_i$, and it
then follows from Definition~\ref{D6.1}(1) that
$\d_{k+1}u_l\in\im\w_i$.

(2) Suppose first that $i=k$; we must show that
$$\d_{k+1}(u_{l-1}\w_k v)\in\im\w_k.$$
But $u_{l-1}\in\im\w_k$ because $l\geq 2$, and $v\in\im\w_k$
because $i=k$, so the result follows from Lemma~\ref{L7.5}.

Now suppose that $k<i<k+l$. We have $u_{k+l-i}\in\im\w_k$
because $i<k+l$, and we have $\d_{k+1}^{i-k-1}v\in\im\w_{k+1}$
by Definition~\ref{D6.1}(1); therefore
$u_{k+l-i+1}\in\im\w_{k+2}$ by Lemma~\ref{L7.6}. We also
have $\d_{k+1}^{i-k-2}v\in\im\w_{k+2}$ by
Definition~\ref{D6.1}(1), so $u_{k+l-i+2}\in\im\w_{k+3}$ by
Lemma~\ref{L7.8}. By repeating the use of
Lemma~\ref{L7.8}, we eventually get $u_l\in\im\w_{i+1}$.
Definition~\ref{D6.1}(1) now gives us $\d_{k+1}u_l\in\im\w_i$.

(3) In this case $u_0\in\im\w_{i-l+1}$ by hypothesis and
$\d_{k+1}^{l-1}v\in\im\w_{i-l+1}$ by Definition~\ref{D6.1}(1);
therefore $u_1\in\im\w_{i-l+2}$ by Lemma~\ref{L7.8}. In
the same way $u_2\in\im\w_{i-l+3}$, and so on. Eventually we
get $u_l\in\im\w_{i+1}$. As before, Definition~\ref{D6.1}(1)
now gives us $\d_{k+1}u_l\in\im\w_i$.
\end{proof}

The final result is a converse to Proposition~\ref{P5.2}.

\begin{proposition} \label{P7.11}
Let $x$ be an element of the form
$$x=\L^r(u_{r-1},\ldots,u_0,\e_r^s w)$$
in a set with complicial identities.

If $i\leq r-3$, if $w\in\im\w_i$, and if $u_p\in\im\w_i$ for
$i+2\leq p<r$, then $x\in\im\w_i$.

If  $w\in\im\e_{r-2}$ and $u_{r-1}\in\im\w_{r-2}$ then
$x\in\im\w_{r-2}$.

If $s\geq 1$ and $u_{r-1}\in\im\w_{r-1}$ then
$x\in\im\w_{r-1}$.

If $r\leq i\leq r+s-2$, and if $u_p\in\im\w_{i-r+p+1}$ for
$0\leq p<r$, then $x\in\im\w_i$.
\end{proposition}

\begin{proof}
Let $v_0=\e_r^s w$, and for $0\leq p<r$ let
$$v_{p+1}=\d_{p+1}(u_p\w_p^{r-p}v_p);$$
thus $x=v_r$. We claim that $v_p\in\im\w_i$ for $0\leq p\leq
r$, except possibly for cases with $i\leq r-3$ and $p=i+2$.

First we suppose that $i\leq r-3$, that $w\in\im\w_i$, and that
$u_p\in\im\w_i$ for $i+2\leq p<r$. We have $v_0\in\im\w_i$ by
Lemma~\ref{L7.7}, and we get
$$v_1\in\im\w_i,\ \ldots,\ v_{i+1}\in\im\w_i$$
by repeated applications of Lemma~\ref{L7.10}(2). We now
have
$$v_{i+3}
 =\d_{i+3}[u_{i+2}\w_{i+2}^{r-i-2}\d_{i+2}(u_{i+1}\w_{i+1}^{r-i-1}v_{i+1})]$$
with $u_{i+2}$~and~$v_{i+1}$ in $\im\w_i$, so
$v_{i+3}\in\im\w_i$ by Lemma~\ref{L7.9}. Since
$u_p\in\im\w_i$ for $i+3\leq p<r$, repeated applications of
Lemma~\ref{L7.10}(1) now show that the elements
$$v_{i+4},v_{i+5},\ldots,v_r$$
are in $\im\w_i$.

Next we suppose that $w\in\im\e_{r-2}$ and
$u_{r-1}\in\im\w_{r-2}$. We have $v_0\in\im\e_{r-2}$ because
$\e_r^s\e_{r-2}=\e_{r-2}\e_{r-1}^s$; hence, by
Definition~\ref{D6.1}(2), $v_0\in\im\w_{r-2}$. As in the
previous case, we get
$$v_1\in\im\w_{r-2},\ \ldots,\ v_{r-1}\in\im\w_{r-2}.$$
Since $v_r=\d_r(u_{r-1}\w_{r-1}v_{r-1})$ and since
$u_{r-1}\in\im\w_{r-2}$, it follows from Lemma~\ref{L7.5}
that $v_r\in\im\w_{r-2}$.

Next we suppose that $s\geq 1$ and $u_{r-1}\in\im\w_{r-1}$.
Then $v_0\in\im\e_r$, and it follows from
Definition~\ref{D6.1}(2) that $v_0\in\im\w_{r-1}$. The rest of
the argument is as in the previous case.

Finally we suppose that $r\leq i\leq r+s-2$ and that
$u_p\in\im\w_{i-r+p+1}$ for $0\leq p<r$. We get
$v_0=\e_{i+1}\e_r^{s-1}w$; therefore $v_0\in\im\w_i$ by
Definition~\ref{D6.1}(2). We then get
$$v_1\in\im\w_i,\ \ldots,\ v_r\in\im\w_i$$
by repeated applications of Lemma~\ref{L7.10}(3).

This completes the proof.
\end{proof}

\section{Freeness} \label{S8}

We have shown that $\Or(-,n)$ is a set with complicial
identities generated by the identity morphism~$\iota_n$ (see
Proposition~\ref{P6.6} and Theorem~\ref{T4.12}). We will now
show that $\Or(-,n)$ is freely generated by~$\iota_n$; that is
to say, given an $n$-dimensional element~$u$ in a set with
complicial identities~$U$, we will show that there is a unique
morphism $f\colon\Or(-,n)\to U$ with $f\iota_n=u$.

We will construct morphisms on $\Or(-,n)$ by combining suitable
functions, called partial morphisms, which are defined on subsets of sets with complicial identities and take their values in sets with complicial identities. 
The domain of a partial morphism is closed under face operations, but not necessarily under degeneracy or wedge operations. A partial morphism preserves the operations so far as can be expected, and is allowed to increase degrees.

\begin{definition} \label{D8.1}
Let $S$ be a subset of a set with complicial identities, let
$U$ be a set with complicial identities, and let $k$ be a
nonnegative integer. Then a \emph{partial morphism of
degree~$k$} from~$S$ to~$U$ is a function $f\colon S\to U$
which increases degrees by~$k$ and  which satisfies the
following conditions.

\textup{(1)} If $x$~is an $m$-dimensional member of~$S$ with
$m>0$ and if $0\leq i\leq m$, then $\d_i x\in S$ and $f\d_i
x=\d_i fx$.

\textup{(2)} If $x$~is a $1$-dimensional member of~$S$ and
$x\in\im\e_0$ then $fx\in\im\e_0$.

\textup{(3)} If $x\in S$ and $x\in\im\w_i$ then $fx\in\im\w_i$.
\end{definition}

Partial morphisms preserve degeneracies and wedges as much as possible.

\begin{proposition} \label{P8.2}
Let $f\colon S\to U$ be a partial morphism and let $x$ be a
member of~$S$.

\textup{(1)} If $x=\e_i y$ for some~$y$, then $y\in S$ and
$fx=\e_i fy$.

\textup{(2)} If $x=y\w_i z$ for some $y$~and~$z$, then $y,z\in
S$ and $fx=fy\w_i fz$.
\end{proposition}

\begin{proof}
We prove these statements in reverse order.

(2) Suppose that $x=y\w_i z$. Then $y=\d_{i+2}x$ and $z=\d_i x$
by Definition~\ref{D6.1}(1); therefore $y,z\in S$. Since also
$fx\in\im\w_i$, it then follows that
$$fx=\d_{i+2}fx\w_i\d_i fx=f\d_{i+2}x\w_i f\d_i x=fy\w_i fz.$$

(1) Suppose that $x=\e_i y$. Then $y=\d_i x$, so that $y\in S$.
We will now prove that $fx=\e_i fy$ by induction on the
dimension of~$x$.

Suppose that $x=\e_0 y$ and $x$~is $1$-dimensional. Then
$fx\in\im\e_0$, and it follows that
$$fx=\e_0\d_0 fx=\e_0 f\d_0 x=\e_0 fy.$$

Suppose that $x=\e_0 y$ and that the dimension of~$x$ is
greater than~$1$. Then $x=\e_0\d_1 y\w_0 y$ by
Definition~\ref{D6.1}(2). From the inductive hypothesis and
from what we have already proved, it now follows that
$$fx
 =f\e_0\d_1 y\w_0 fy
 =\e_0 f\d_1 y\w_0 y
 =\e_0\d_1 fy\w_0 fy
 =\e_0 fy.$$

Finally suppose that $x=\e_i y$ with $i>0$. Then
$x=y\w_{i-1}\e_{i-1}\d_{i-1}y$ by Definition~\ref{D6.1}(2), and
it follows as in the previous case that $fx=\e_i fy$.

This completes the proof.
\end{proof}

In certain circumstances, it follows that partial morphisms also preserve expressions of the form $\L^r(u_{r-1},\ldots,u_0,v)$.

\begin{proposition} \label{P8.3}
Let $f$ be a partial morphism on a subset~$S$ of $\Or(-,n)$,
let
$$x=\L^r(u_{r-1},\ldots,u_0,v)$$
be defined in $\Or(-,n)$, let the terminus of~$v$ be~$t$, and
let the corank of~$v$ be~$s$. Suppose that $S$~contains all
morphisms with terminus~$t$\textup{;} alternatively, suppose
that $\rank v\geq r$ and that $S$~contains all morphisms with
terminus~$t$ and corank~$s$. Then the morphisms
$x,u_{r-1},\ldots,u_0,v$ are members of~$S$, and
$$fx=\L^r(fu_{r-1},\ldots,fu_0,fv).$$
\end{proposition}

\begin{proof}
This follows from Propositions \ref{P5.3} and~\ref{P8.2}.
\end{proof}

Given an $n$-dimensional element~$u$ in a set with complicial identities, we will now show how to construct a morphism on $\Or(-,n)$ sending~$\iota_n$ to~$u$. We use a  recursion on terminus and corank. For $t\geq 0$ let $S_t$ be
the subset of $\Or(-,n)$ consisting of the morphisms of
terminus at most~$t$. For $t\geq 0$ and $s\geq -1$ let $S_t^s$
be the subset of~$S_t$ consisting of the following
morphisms: the morphisms with terminus less than~$t$; the
morphism $\d_0^t\d_{t+1}^{n-t}\iota_n$; the $t$-cones, in the
case that $t>0$; the morphisms with terminus~$t$ and with
corank at most~$s$. We now proceed as follows. We construct a partial morphism  on~$S_0^{-1}$ sending $\d_1^n\iota_n$ to~$u$, we extend it to~$S_0^0$, $S_0^1$, \dots, and eventually to the whole of~$S_0$, we then construct a partial morphism on~$S_1^{-1}$ sending $\d_2^{n-1}\iota_n$ to~$u$, and so on, finishing with a partial morphism on~$S_n=\Or(-,n)$ (actually a genuine morphism) sending~$\iota_n$ to~$u$. The steps in this process are given by Lemmas \ref{L8.4}--\ref{L8.6}.

\begin{lemma} \label{L8.4}
Let $u$ be an element in a set with complicial identities~$U$.
Then there is a partial morphism $f\colon S_0^{-1}\to U$ such
that $f\d_1^n\iota_n=u$.
\end{lemma}

\begin{proof}
The only morphism in~$S_0^{-1}$ is the zero-dimensional
morphism $\d_1^n\iota_n$. The condition $f\d_1^n\iota_n=u$
therefore defines a function $f\colon S_0^{-1}\to U$, and this
function is trivially a partial morphism.
\end{proof}

\begin{lemma} \label{L8.5}
For $t>0$, let $F\colon S_{t-1}\to U$ be a partial morphism of
degree $k>0$. Then there is a partial morphism $f\colon
S_t^{-1}\to U$ of degree $k-1$ such that
$f\d_{t+1}^{n-t}\iota_n=F\d_t^{n-t+1}\iota_n$.
\end{lemma}

\begin{proof}
We define a function $f\colon S_t^{-1}\to U$ as follows: if
$x$~is a member of $\Or(m,n)$ with terminus less than~$t$, then
$$fx=\d_{m+1}Fx;$$
if $x=\d_0^t\d_{t+1}^{n-t}\iota_n$ then
$$fx=\d_0^t F\d_t^{n-t+1}\iota_n;$$
if $x$~is a $t$-cone in $\Or(m+1,n)$ then $\d_{m+1}x$ has
terminus less than~$t$, and we make the definition
$$fx=F\d_{m+1}x.$$
It is clear that $f$~increases degrees by $k-1$ and that
$$f\d_{t+1}^{n-t}\iota_n=F\d_t\d_{t+1}^{n-t}\iota_n=F\d_t^{n-t+1}\iota_n;$$
it therefore remains to verify conditions (1)--(3) of
Definition~\ref{D8.1}.

(1) Let $\d_i x$ be a face of an $m$-dimensional member
of~$S_{t-1}$, so that $m>0$ and $0\leq i\leq m$. Then $\d_i x$
is an $(m-1)$-dimensional member of~$S_{t-1}$, so that $\d_i
x\in S_t^{-1}$, and
$$f\d_i x=\d_m F\d_i x=\d_m\d_i Fx=\d_i\d_{m+1}Fx=\d_i fx.$$

The zero-dimensional morphism $\d_0^t\d_{t+1}^{n-t}\iota_n$
does not have any faces.

Let $x$ be a $1$-dimensional $t$-cone, and consider the face
$\d_0 x$. This is the zero-dimensional morphism
$\d_0^t\d_{t+1}^{n-1}\iota_n$; therefore $\d_0 x\in S_t^{-1}$
and $f\d_0 x=\d_0^t F\d_t^{n-t+1}\iota_n$. The zero-dimensional
morphism $\d_1 x$ must have the form $[0]\mapsto [j]$ with
$0\leq j\leq t-1$, and it can be expressed as
$\d_0^j\d_{j+1}^{t-j-1}\d_t^{n-t+1}\iota_n$. Therefore
$$\d_0 fx
 =\d_0 F\d_1 x
 =\d_0\d_0^j\d_{j+1}^{t-j-1}F\d_t^{n-t+1}\iota_n
 =\d_0^t F\d_t^{n-t+1}\iota_n,$$
and it follows that $f\d_0 x=\d_0 fx$.

Let $\d_i x$ be a face of an $(m+1)$-dimensional $t$-cone such
that $m>0$ and $0\leq i\leq m$. Then $\d_i x$ is an
$m$-dimensional $t$-cone, so that $\d_i x\in S_t^{-1}$ and
$$f\d_i x=F\d_m\d_i x=F\d_i\d_{m+1}x=\d_i F\d_{m+1}x=\d_i fx.$$

Finally, let $x$ be an $(m+1)$-dimensional $t$-cone with $m\geq
0$ and consider the face $\d_{m+1}x$. This is an
$m$-dimensional member of~$S_{t-1}$; therefore $\d_{m+1}x\in
S_t^{-1}$ and
$$f\d_{m+1}x=\d_{m+1}F\d_{m+1}x=\d_{m+1}fx.$$

(2) Let $x$ be a $1$-dimensional member of~$S_t^{-1}$ lying in
the image of~$\e_0$. Then $x\in S_{t-1}$, so that $fx=\d_2 Fx$,
and we also have $x=\e_0\d_0 x$. It now follows from
Proposition~\ref{P8.2} that $fx\in\im\e_0$, because
$$fx=\d_2 F\e_0\d_0 x=\d_2\e_0 F\d_0 x=\e_0\d_1 F\d_0 x.$$

(3) Let $x$ be an $m$-dimensional member of~$S_{t-1}$ lying in
the image of~$\w_i$. Then $fx=\d_{m+1}Fx$. We also have
$x=\d_{i+2}x\w_i\d_i x$ by Definition~\ref{D6.1}(1), and we
must have $0\leq i\leq m-2$. It now follows from
Proposition~\ref{P8.2} and Definition~\ref{D6.1}(1) that
$fx\in\im\w_i$, because
$$fx
 =\d_{m+1}F(\d_{i+2}x\w_i\d_i x)
 =\d_{m+1}(F\d_{i+2}x\w_i F\d_i x)
 =\d_m F\d_{i+2}x\w_i\d_m F\d_i x.$$

The zero-dimensional morphism $\d_0^t\d_{t+1}^{n-t}\iota_n$
cannot belong to the image of a wedge operation.

Let $x$ be an $(m+1)$-dimensional $t$-cone, and suppose that
$x\in\im\w_i$ with $0\leq i\leq m-2$. Then $fx\in\im\w_i$ as
before, because
$$fx
 =F\d_{m+1}(\d_{i+2}x\w_i\d_i x)
 =F(\d_m\d_{i+2}x\w_i\d_m\d_i x)
 =F\d_m\d_{i+2}x\w_i F\d_m\d_i x.$$

Finally, suppose that $x$~is an $(m+1)$-dimensional $t$-cone
lying in the image of~$\w_{m-1}$. We must have
$x=\e_{m-1}\d_{m-1}x$ (see Propositions \ref{P3.8}
and~\ref{P3.4}). It now follows from Proposition~\ref{P8.2} and
Definition~\ref{D6.1}(2) that $fx\in\im\w_{m-1}$, because
$$fx
 =F\d_{m+1}\e_{m-1}\d_{m-1}x
 =F\e_{m-1}\d_m\d_{m-1}x
 =\e_{m-1}F\d_m\d_{m-1}x.$$

This completes the proof.
\end{proof}

\begin{lemma} \label{L8.6}
For $t\geq 0$ and $s\geq 0$, if $f$~is a partial morphism
on~$S_t^{s-1}$, then there is a partial morphism~$f'$
on~$S_t^s$ which is an extension of~$f$.
\end{lemma}

\begin{proof}
Let $x$ be a morphism in~$S_t^s$ with terminus less than~$t$ or
with terminus~$t$ and corank less than~$s$; then $x$~belongs
to~$S_t^{s-1}$ and we make the definition
$$f'x=fx.$$

Now let $x$ be a morphism with terminus~$t$ and corank~$s$, and
let the rank of~$x$ be~$r$. According to Theorem~\ref{T4.11},
$x$~has a canonical form given by
$$x=\L^r(\alpha_{r-1}x,\ldots,\alpha_0 x,\e_r^s\gamma x).$$
According to Propositions \ref{P4.8} and~\ref{P4.6}, the
morphisms $\alpha_p x$ and $\gamma x$ are members
of~$S_t^{s-1}$, and we will obtain $f'x$ by applying~$f$ to
each of these morphisms.

In order to show that this is possible, we use
Proposition~\ref{P7.2}. For $0\leq p<r$ it follows from this
proposition and from the existence of the canonical form that
$$\d_p\alpha_p x
 =\L^p(\alpha_{p-1}x,\ldots,\alpha_0 x,\d_{p+1}^{r-p}\e_r^s\gamma x).$$
If $s=0$ then the argument $\d_{p+1}^{r-p}\e_r^s\gamma x$ has
terminus less than~$t$; if $s>0$ then it has terminus~$t$, rank
$p+1$ and corank $s-1$; in both cases, it follows from
Proposition~\ref{P8.3} that
$$\d_p f\alpha_p x
 =f\d_p\alpha_p x
 =\L^p(f\alpha_{p-1}x,\ldots,f\alpha_0 x,
 f\d_{p+1}^{r-p}\e_r^s\gamma x).$$
If now $s=0$ then
$$f\d_{p+1}^{r-p}\e_r^s\gamma x
 =f\d_{p+1}^{r-p}\gamma x
 =\d_{p+1}^{r-p}f\gamma x
 =\d_{p+1}^{r-p}\e_r^s\gamma x;$$
if $s>0$ then $\e_r^{s-1}\gamma x\in S_t^{s-1}$ and we get
$$f\d_{p+1}^{r-p}\e_r^s\gamma x
 =f\d_{p+1}^{r-p-1}\e_r^{s-1}\gamma x
 =\d_{p+1}^{r-p-1}\e_r^{s-1}f\gamma x
 =\d_{p+1}^{r-p}\e_r^s f\gamma x;$$
in any case we get
$$\d_p f\alpha_p x
 =\L^p(f\alpha_{p-1}x,\ldots,f\alpha_0 x,
 \d_{p+1}^{r-p}\e_r^s f\gamma x).$$
Because of Proposition~\ref{P7.2}, we can now define $f'x$ by
the formula
$$f'x=\L^r(f\alpha_{r-1}x,\ldots,f\alpha_0
 x,\e_r^s f\gamma x).$$

We will now show that $f'$~is an extension of~$f$; that is to say, we will show that $f'x=fx$ when $x$~is a morphism with terminus~$t$ and corank~$s$ belonging to~$S_t^{s-1}$. This situation can occur only when $s=0$, with  
$x=\d_0^t\d_{t+1}^{n-t+1}\iota_n$ or with $x$ a $t$-cone. In these cases, let $r$ be the rank of~$x$; then it follows from
Example~\ref{E4.7} that
$$f'x=\L^r(f\e_{r-1}\d_r x,\ldots,f\e_0\d_1^r x,fx),$$
it follows from Proposition~\ref{P8.2} that
$$f'x=\L^r(\e_{r-1}\d_r fx,\ldots,\e_0\d_1^r fx,fx),$$
and it follows from Proposition~\ref{P7.4} that $f'x=fx$ as
required.

It remains to show that $f'$~is a partial morphism by verifying
the conditions of Definition~\ref{D8.1}. It clearly suffices to
consider morphisms with terminus~$t$ and corank~$s$, and we
argue as follows.

(1) Let $x$ be a morphism with terminus~$t$, with corank~$s$
and with rank~$r$, and consider a face $\d_i x$ with $0\leq
i<r$. In this case $\d_i x$ has terminus~$t$, corank~$s$ and
rank $r-1$, so that $\d_i x\in S_t^s$ and it follows from
Proposition~\ref{P5.1} that
$$f'\d_i x
 =\L^{r-1}(f\d_i\alpha_{r-1}x,\ldots,f\d_i\alpha_{i+1}x,
 f\alpha_{i-1}x,\ldots,f\alpha_0 x,\e_{r-1}^s f\d_i\gamma x).$$
Since $f\d_i\alpha_p x=\d_i f\alpha_p x$ for $i<p<r$ and since
$$\e_{r-1}^s f\d_i\gamma x
 =\e_{r-1}^s\d_i f\gamma x
 =\d_i\e_r^s f\gamma x,$$
it now follows from Proposition~\ref{P7.3} that
$$f'\d_i x
 =\d_i\L^r(f\alpha_{r-1}x,\ldots,f\alpha_0 x,\e_r^s f\gamma x)
 =\d_i f'x.$$

Now let $x$ be a morphism with terminus~$t$, with corank~$s$
and with rank zero, and consider the face $\d_0 x$. This exists
only in cases with $s>0$. It has terminus~$t$, corank $s-1$ and
rank zero, and it therefore belongs to~$S_t^{s-1}$. It follows
that $\d_0 x\in\ S_t^s$. It also follows from
Theorem~\ref{T4.11} and Proposition~\ref{P7.3} that
$$f'\d_0 x
 =f\d_0 x
 =f\d_0\L^0(\e_0^s\gamma x)
 =f\d_0\e_0^s\gamma x
 =f\e_0^{s-1}\gamma x
 =\e_0^{s-1}f\gamma x$$
and that
$$\d_0 f'x
 =\d_0\L^0(\e_0^s f\gamma x)
 =\d_0\e_0^s f\gamma x
 =\e_0^{s-1}f\gamma x;$$
therefore $f'\d_0 x=\d_0 f' x$.

Now let $x$ be a morphism with terminus~$t$, with corank~$s$,
and with rank $r>0$, and consider the face $\d_r x$. If $s=0$
then $\d_r x$ has terminus less than~$t$; if $s>0$ then $\d_r
x$ has terminus~$t$ and corank $s-1$; in both cases, $\d_r x\in
S_t^{s-1}$. It follows that $\d_r x\in S_t^s$. It also follows
from Theorem~\ref{T4.11} and Proposition~\ref{P7.3} that
$$ f'\d_r x
 =f\d_r x
 =f\d_r\L^r(\alpha_{r-1}x,\ldots,\alpha_0 x,\e_r^s\gamma x)
 =f\d_r\alpha_{r-1}x
 =\d_r f\alpha_{r-1}x$$
and
$$\d_r f'x
 =\d_r\L^r(f\alpha_{r-1}x,\ldots,f\alpha_0 x,\e_r^s f\gamma x)
 =\d_r f\alpha_{r-1}x;$$
therefore $f'\d_r x=\d_r f'x$.

Finally, let $x$ be a morphism with terminus~$t$, with
corank~$s$, and with rank~$r$, and consider $\d_i x$ for
$r<i\leq r+s$. In this case $\d_i x$ has terminus~$t$ and
corank $s-1$, so that $\d_i x\in S_t^{s-1}$; therefore $\d_i
x\in S_t^s$. It now follows from Theorem~\ref{T4.11} and
Proposition~\ref{P7.3} that
\begin{align*}
 f'\d_i x
 &=f\d_i x\\
 &=f\d_i\L^r(\alpha_{r-1}x,\ldots,\alpha_0 x,\e_r^s\gamma x)\\
 &=f\L^r(\d_i\alpha_{r-1}x,\ldots,\d_{i-r+1}\alpha_0 x,\d_i\e_r^s\gamma x)\\
 &=f\L^r(\d_i\alpha_{r-1}x,\ldots,\d_{i-r+1}\alpha_0 x,\e_r^{s-1}\gamma x).
\end{align*}
Here $\e_r^{s-1}\gamma x$ has terminus~$t$, rank~$r$ and corank
$s-1$, so it follows from Propositions \ref{P8.3}
and~\ref{P8.2} that
\begin{align*}
 f'\d_i x
 &=\L^r(f\d_i\alpha_{r-1}x,\ldots,f\d_{i-r+1}\alpha_0 x,f\e_r^{s-1}\gamma x)\\
 &=\L^r(\d_i f\alpha_{r-1}x,\ldots,\d_{i-r+1}f\alpha_0 x,\e_r^{s-1}f\gamma x).
\end{align*}
On the other hand, it follows from Proposition~\ref{P7.3} that
\begin{align*}
 \d_i f'x
 &=\L^r(\d_i f\alpha_{r-1}x,\ldots,\d_{i-r+1}f\alpha_0 x,\d_i\e_r^s f\gamma x)\\
 &=\L^r(\d_i f\alpha_{r-1}x,\ldots,\d_{i-r+1}f\alpha_0 x,\e_r^{s-1}f\gamma x);
\end{align*}
therefore $f'\d_i x=\d_i f'x$.

(2) Suppose that $x$~is a $1$-dimensional morphism with
terminus~$t$ and positive corank~$s$ lying in the image
of~$\e_0$. This can occur only in the case that $s=1$, and the
rank of~$x$ must be zero; therefore $f'x=\L^0(\e_0 f\gamma
x)=\e_0 f\gamma x$. It follows that $f'x$ is in the image
of~$\e_0$.

(3) Suppose that $x$~is a morphism with terminus~$t$ and
corank~$s$ lying in the image of~$\w_i$. Let the rank of~$x$
be~$r$, so that $0\leq i\leq r+s-2$. It follows from
Proposition~\ref{P5.2} that certain morphisms $\alpha_p x$ are
in certain sets $\im\w_j$, and it may also follow that $\gamma
x$ is in $\im\w_i$ or $\im\e_i$, The images $f\alpha_p x$ and
$f\gamma x$ then satisfy the same conditions, and it follows
from Proposition~\ref{P7.11} that $f'x\in\im\w_i$.

This completes the proof.
\end{proof}

From Lemmas \ref{L8.4}--\ref{L8.6} we get the main result.

\begin{theorem} \label{T8.7}
For $n\geq 0$, the graded set $\Or(-,n)$ is the set with
complicial identities freely generated by the identity
morphism~$\iota_n$ in $\Or(n,n)$.
\end{theorem}

\begin{proof}
First, by Proposition~\ref{P6.6}, $\Or(-,n)$ is a set with
complicial identities.

Now let $u$ be an $n$-dimensional element in a set with
complicial identities~$U$. We must show that there is a unique
morphism from $\Or(-,n)$ to~$U$ sending~$\iota_n$ to~$u$.

We construct a suitable morphism as follows. By
Lemma~\ref{L8.4} there is a partial morphism on~$S_0^{-1}$
sending $\d_1^n\iota_n$ to~$u$; by repeated applications of
Lemma~\ref{L8.6}, there is a partial morphism on the entire
set~$S_0$ sending $\d_1^n\iota_n$ to~$u$; by Lemma~\ref{L8.5},
there is a partial morphism on~$S_1^{-1}$ sending
$\d_2^{n-1}\iota_n$ to~$u$; by Lemma~\ref{L8.6} there is a
partial morphism on the entire set~$S_1$ sending
$\d_2^{n-1}\iota_n$ to~$u$; etc. Since $\Or(-,n)=S_n$, we eventually obtain a partial
morphism $f\colon\Or(-,n)\to U$ such that $f\iota_n=u$. By
Proposition~\ref{P8.2}, $f$~is in fact a morphism of sets with
complicial identities. Since, according to Theorem~\ref{T4.12},
$\Or(-,n)$ is generated by~$\iota_n$, it follows that $f$~is
the only morphism from $\Or(-,n)$ to~$U$ with $f\iota_n=u$.

This completes the proof.
\end{proof}

\end{document}